\documentclass[11pt]{article}
\usepackage[T1]{fontenc}
\usepackage[font=small,labelfont=bf,tableposition=top]{caption}

\usepackage[psamsfonts]{amssymb}
\usepackage{amsmath,amsfonts}
\usepackage{amsthm}
\usepackage{graphicx}
\usepackage{bbm}

\usepackage[T1]{fontenc}

\DeclareCaptionLabelFormat{andtable}{#1~#2  \&  \tablename~\thetable}

\newtheorem{Theorem}{Theorem}
\newtheorem{Lemma}{Lemma}
\newtheorem{Corollary}{Corollary}

\newmuskip\pFqskip
\mathchardef\pFcomma=\mathcode`, 

\newcommand*\pFq[5]{%
  \begingroup
  \begingroup\lccode`~=`,
    \lowercase{\endgroup\def~}{\pFcomma\mkern\pFqskip}%
  \mathcode`,=\string"8000
  {}_{#1}F_{#2}\biggl[\genfrac..{0pt}{}{#3}{#4};#5\biggr]%
  \endgroup
}

\begin{document}
\title{The higher order asymptotic expansion of the Krawtchouk polynomials\footnote{Supported by RFBR (grant 14-01-00373).}}
\author{Aleksei Minabutdinov\thanks{National Research University Higher School of Economics (HSE), Department of Applied Mathematics and Business Informatics, 3A Kantemirovskaya ul.,  St.Petersburg, 194100,  Russia }}
\maketitle
\begin{abstract}
The paper extends the classical result on the convergence of the Krawtchouk polynomials to the  Hermite polynomials. We provide the uniform asymptotic expansion in terms of the Hermite polynomials. We explicitly obtain expressions for a few initial terms of this expansion.
The research is motivated by the study of  ergodic sums of the Pascal adic transformation.
\end{abstract}


\section{Introduction}
Let $0<p<1$, $q=1-p$ and let $N$ be a positive integer. The non-normalized Krawtchouk polynomials of variable $x$  can be defined by the identity
\begin{equation}K_{n}(x,p, N) = \pFq{2}{1}{-x, -n}{-N}{\frac{1}{p}},\end{equation}
where $x$ and $n$ are in $\{ 0,1,\dots N\}$  and $_2F_1$ is the Gauss hypergeometric function.
The normalized Krawtchouk polynomials are usually defined as follows:
\begin{equation}\label{eq:KrawtchoukDef}k_n^{(p)}(x,N)=(-p)^n\binom{N}{n}K_{n}(x,p, N).  \end{equation}
The second argument $N$ is usually omitted, so we write  $k_n^{(p)}(x)$ instead of $k_n^{(p)}(x,N)$.
They form an orthogonal system on the discrete set $\{0, 1, 2, ..., N\}$ with
weight function
\[\rho(x) = \frac{N!\,p^xq^{N-x}}{\Gamma(1+x)\Gamma(N+1-x)}\]
and  orthogonality relation
\[\sum\limits_{x=0}^{N}k_i^{(p)}(x) k_j^{(p)}(x)\rho(x) = \binom{N}{j}(pq)^j\delta_{ij}, \quad i,j = 0,1\dots, N.\]

Krawtchouk polynomials satisfy the following Rodrigues-type formula (see \cite{NikiforovSuslovUvarov}, Section 2, (22a))

\begin{equation}
\label{eq:RodriguesTypeForKrawthcouk}
k_n^{(p)}(x) = \frac{(-q)^n}{n!}\frac{\Delta^n\big(\rho(x)x^{\underline{n}}\big)}{\rho(x)},
\end{equation}
where $\Delta f(x) = f(x+1)-f(x)$ and $y^{\underline{k}} = y(y-1)\dots(y-k+1)$. Formula \eqref{eq:RodriguesTypeForKrawthcouk} is also often taken as the definition of the Krawtchouk polynomials, see e.g.~\cite{Sharapudinov}. Expression \eqref{eq:RodriguesTypeForKrawthcouk} also shows that the Krawtchouk polynomial $k_n^{(p)}(\cdot)$ can be considered as an analytic function on $[0,N].$

Finally, it was recently found in \cite{LodkinMinabutdinov} that the  Krawtchouk polynomial $(-2p)^nK_{n}(k,p, N)$ has a natural interpretation as the ergodic sum along tower $\tau_{N,k}$ of the Pascal adic transformation  for the  \emph{orthogonalized\footnote{with respect to the Bernoulli $(p,q)$ measure}  Walsh-Paley function} $w_t^q$ where the sum of binary digits  in the binary representation of  $t\in\mathbb{N}$   equals to $n$ (see \cite{LodkinMinabutdinov} for details).


There is a classical result on convergence of the (properly renormalized) Krawtchouk polynomials to the Hermite polynomials (see \cite{Krawtchouk1929}):
\begin{equation}
\label{eq:classicalApproximOfKrawByHermite}
\lim\limits_{N\rightarrow\infty}\Big(\frac{2}{Npq}\Big)^{n/2}n!\,k_n^{(p)}(\hat{x}) =H_n(x), \end{equation}
with $\hat{x} = Np+(2Npq)^{1/2}x$ and $H_n(x) = (-1)^ne^{x^2}\frac{d^n}{dx^n}e^{-x^2}$.
The result was extended by Sharapudinov in \cite{Sharapudinov}, where he obtained the asymptotic formula
\begin{equation}
\label{eq:resultSharapudinov}
\begin{gathered}
  (2Npq \pi n!)^{1/2}(Npq)^{-n/2}\rho(\hat{x})e^{x^2/2}k_n^{(p)}(\hat{x}) = \\ \quad\quad \quad\quad  \quad\quad\quad\quad \quad\quad \quad=e^{-x^2/2}(2^nn!)^{-1/2}H_n(x)+O(n^{7/4}N^{-1/2}),
\end{gathered}
\end{equation}
with $\hat{x} = Np+(2Npq)^{1/2}x,\  n=O(N^{1/3}), \ x = O(n^{1/2}).$

Let $A$ be a positive real.
In this paper we are interested in the higher order uniform in $v\in[-A\sqrt{N},A{\sqrt{N}}]$ asymptotic expansion of the Krawtchouk polynomials of the form:
\begin{equation} \label{eq:aympExpansion}
k_n^{(p)}(\hat{x}) =\sum_{j=0}^{M}c_{j+1}(v)N^{[n/2]-j}+o(N^{[n/2]-M}),\end{equation}
where $\hat{x} = Np+v,\  n=O(1), \ M\in\mathbb{N}\cup\{0\}$ and $[t]$ gives the integer part of a real number $t$. Our main result is the uniform asymptotic expansion of $\rho(\hat{x}) k_n^{(p)}(\hat{x})$ in terms of the Hermite polynomials is stated by Theorem~\ref{Th:asympInTermsOfHermite}. Our approach is based on the result by V.~V.~Petrov (\cite{Petrov}) extending the Local Limit Theorem (LLT).

We are especially interested in the expression for the first non-constant term $c_j(v)$ in expansion \eqref{eq:aympExpansion}, due to the study of ergodic sums of the Pascal adic transformation in \cite{LodkinMinabutdinov}. It turns out that the value of $j$ depends on the parity of index $n$. In Corollary \ref{Cor:asympSimplified} under the additional assumption $v = o({N}^{1/3})$ we obtain elegant explicit expressions for $c_1(v)$ and $c_2(v)$ for odd and even values of $n$ correspondingly.

Recently, there is a considerable interest in the asymptotics of the  Kraw-tchouk polynomials, when the parameter  $N$ grows to infinity (see e.g. \cite{DaiWong} and references therein).  In paper \cite{DaiWong}  authors considered (among many others) the case when $x=O(1)$ while $N\rightarrow \infty.$ Nevertheless they ruled out the case $n\approx Np$.  We mention, that the self-duality relation $K_{x}(n,p, N)= K_{n}(x,p, N)$  implies that the one-term asymptotic expansion follows already from \eqref{eq:classicalApproximOfKrawByHermite} in this case.

The author is grateful to Profs. A. A. Lodkin and A. M. Vershik for their advices and kind support.



\section{Main results}
Using identity $\Delta^sx^{\underline{n}}= n^sx^{\underline{n-s}}$ we  rewrite  Rodrigues  formula \eqref{eq:RodriguesTypeForKrawthcouk} as follows
\begin{equation}\label{eq:RodriguesRewrite}
\begin{gathered}
\rho(x)k_n^{(p)}(x) = \frac{(-q)^n}{n!} \Delta^n\big[\rho(x)x^{\underline{n}}\big] =  \frac{(-q)^n}{n!}\sum\limits_{k=0}^n\binom{n}{k}\Delta^{n-k}\rho(x)\,\Delta^k(x+n-k)^{\underline{n}} = \\
=  \frac{(-q)^n}{n!} \sum\limits_{k=0}^n\binom{n}{k}n^{\underline{k}}\,(x+n-k)^{^{\underline{(n-k)}}}
\Delta^{n-k}\rho(x).
\end{gathered}
\end{equation}
Next we consider the term $\Delta^{s}\rho(x), \ s\geq0,$ separately. We introduce $h$-step forward difference operator   $\Delta_h f(x) =f(x+h)-f(x),\,\ \Delta_h^nf(x) = \Delta_h(\Delta_h^{n-1}f(x)), n\geq2.$ Following \cite{Sharapudinov} we let $h$ to be equal to $\frac{1}{\sqrt{2Npq}}$, thus we can write\footnote{Note that $\Delta$ is  forward difference in $\hat{x}$, while $\Delta_h$ is $h$-step forward difference in $x$.}
\begin{equation}\label{eq:Delta_h_rewrite}\Delta^{s}\rho( \hat{x}) =\Delta^{s}_h\, \rho(\hat{x}(x)),\end{equation}
where $\hat{x}=Np+(2Npq)^{1/2}x$.
In fact, asymptotic relation \eqref{eq:classicalApproximOfKrawByHermite}, as well as \eqref{eq:resultSharapudinov}, follows\footnote{Other approaches are based on the convergence of a the  difference equation, which polynomial solutions define the Krawtchouk polynomials, to the differential equation defining Hermite polynomials, see details  in e.g. \cite{NikiforovSuslovUvarov}, or on the  convergence  of generating functions, see \cite{Szego}.} from the LLT applied to $\rho(\hat{x})$ and the mean-value theorem $\Delta^{n}_hf(x) = h^n\frac{d^n}{dx^n}f(x+nh\theta),\, \theta\in(0,1),$ combined with the proper estimation of the residual term (see details in \cite{Sharapudinov}).
 In order to obtain higher order approximation we consider function $\rho(x)$ as a probability of $x$  successes in a sequence of $N$ independent $(p,q)$-Bernuolli trials and use Theorem 16  from \S 3 of the work \cite{Petrov} by V.~V.~Petrov for the Bernoulli distribution\footnote{In book \cite{Petrov} a slightly different definition of the Hermite polynomials was used: $\text{He}_n(x) = (-1)^ne^{x^2/2}\frac{d^n}{dx^n}e^{-x^2/2}.$ They are related with polynomials $H_n(x)$ by the equality $\text{He}_n(x)= 2^{-n/2}H_n\big(\frac{x}{\sqrt{2}}\big)$. }:

\begin{Theorem}
\label{Th:Petrov}
Let $M$ be a nonnegative integer and $\sigma^2 =pq$. The following asymptotic expansion holds uniformly in $x$ such that $Np+(2Npq)^{1/2}x\in\mathbb{Z}$:
\begin{equation} \label{eq:PetrovTheorem}(1+|x|^{M+2})\bigg(\sqrt{N}\rho(\hat{x}) -  \frac{1}{\sqrt{2\pi}\sigma}e^{-x^2}\, \sum\limits_{\nu=0}^{M}\frac{\tilde{q}_{\nu}(x)}{N^{\nu/2}}\bigg)=o\Big(\frac{1}{N^{M/2}}\Big),\end{equation}
where $\hat{t}=Np+(2Npq)^{1/2}t$ and functions $\tilde{q}_\nu$ are defined as follows:
\begin{equation}\label{eq:q_nuDef}\tilde{q}_\nu(x) =\sum \frac{1}{2^{(\nu/2+s)}}H_{\nu+2s}(x)\prod\limits_{m=1}^\nu\frac{1}{k_m!}\Big(\frac{\gamma_{m+2}}{(m+2)!\sigma^{m+2}}\Big)^{k_m }, \end{equation}
with $\gamma_i, \ i\geq0,$ being cumulants of the Bernoulli $(p,q)$ distribution,  and the summation in the right-hand-side  taken over all nonnegative solutions $(k_1, k_2,\dots,k_\nu) $ of the equation $k_1+2k_2+\dots+\nu k_\nu=\nu $ such that $s = k_1+k_2+\dots+ k_\nu.$
\end{Theorem}

To obtain the higher order approximation for $\Delta^{s}_h$ we use the formal representation $\Delta_h = e^{hD}-1$ where $D=\frac{d}{dx}$ is the difference operator which yields
\begin{equation}\label{eq:DeltaSeries}\Delta^{s}_h =(e^{hD}-1)^s = \sum\limits_{i=s}^{\infty} a_{s,i-s} (hD)^i.\end{equation}

For any nonnegative integer $K$ and any analytic real-valued function $f$ we can truncate the series and  simply write 
\begin{equation} \label{eq:s-diff_operator_asymp}\Delta^{s}_hf(x) = \sum\limits_{i=0}^{K} a_{s,i}\, D^{s+i}f(x)h^{s+i}+O(h^{K+s+1}).\end{equation}
Coefficients $a_{s,j}$ could be found by means of the Multinomial theorem as follows:
 \begin{equation}\label{eq:forward_h_diff_asymp}a_{s,j}=\sum s! \prod\limits_{r=1}^{j+1}\frac{1}{k_r!}\Big(\frac{1}{r!}\Big)^{k_r},\end{equation}
 where $s$ and $j $ are nonnegative integers and  the summation in the right-hand-side of \eqref{eq:forward_h_diff_asymp} is taken over all nonnegative solutions $(k_1, k_2,\dots,k_j) $ of the equation  $k_1+2k_2+\dots+j k_j=j $ satisfing $k_1+k_2+\dots+k_j=s $. In particular we have
 \begin{equation}
 \label{ex:a_ceof}a_{s,0}=1,\, a_{s,1}=\frac{s}{2},\, a_{s,2} = \frac{s(3s+1)}{24} \text{ and so on}\dots\end{equation}

Theorem \ref{Th:Petrov}  suggests that in order to obtain asymptotic expansion for $\Delta^{s}_h\, \rho(\hat{x}) $ we need the asymptotic expansions for $\Delta^{s}_{h\,}\big(e^{-x^2}\tilde{q}_\nu(x)\big) ,\ \nu\geq0$.
We denote by $b_{\nu,s}$ coefficients $\frac{1}{2^{(\nu/2+s)}}\prod\limits_{m=1}^\nu\frac{1}{k_m!}\Big(\frac{\gamma_{m+2}}{(m+2)!\sigma^{m+2}}\Big)^{k_m }$ arising in the right-hand-side of formula \eqref{eq:q_nuDef}. Thus we have $\tilde{q}_\nu(x) = \sum b_{\nu,s}H_{\nu+2s}(x)$. We denote by $\tilde{g}_{\nu,r}(x)$ the following expression $e^{x^2}\frac{d^r}{dx^r}e^{-x^2}\tilde{q}_\nu(x)$. For $e^{x^2}\frac{d^s}{dx^s}\big(e^{-x^2}H_{n}(x)\big)$ we have the identity \begin{equation}
\label{eq:DExp(x)Hn}
e^{x^2}\frac{d^s}{dx^s}\big(e^{-x^2}H_{n}(x)\big) =(-1)^ne^{x^2}\frac{d^s}{dx^s}\big(\frac{d^n}{dx^n}e^{-x^2} \big) = (-1)^sH_{n+s}(x), \end{equation}
which implies
\[\tilde{g}_{\nu,r}(x)=e^{x^2}\frac{d^r}{dx^r}e^{-x^2}\tilde{q}_\nu(x)= \sum (-1)^rb_{\nu,s}H_{\nu+2s+r}(x), \]
where limits of summation in the right-hand-side are the same  as in \eqref{eq:q_nuDef}. Let $A$ be a positive real and  $r$ be a nonnegative integer. In Appendix we show that we can differentiate\footnote{For   $s=0$ we use a convention $\frac{d^s}{dx^s}f(x)\equiv f(x)$.} asymptotic expansion \eqref{eq:PetrovTheorem} and the estimate of the residual is valid uniformly for \emph{any}\footnote{Note that Theorem \ref{Th:Petrov} states this for $r=0$ and values of $x$ from the discrete set only.} $x$ from the set $[-A,A]$:
\begin{equation}\label{eq:PetrovTheoremAfterDiff}\frac{d^r}{dx^r}\sqrt{N}\rho(\hat{x}(x)) =   \frac{1}{\sqrt{2\pi}\sigma}e^{-x^2}\, \sum\limits_{\nu=0}^{M}\frac{\tilde{g}_{\nu,r}(x)}{N^{\nu/2}}+o\Big(\frac{1}{N^{M/2}}\Big).\end{equation}
It is instructive to compare the above result with Theorem 7, Section 6, in~\cite{Petrov}.

\emph{Remark $1$.} Since parameter $M$ in Theorem \ref{Th:Petrov} and formula \eqref{eq:PetrovTheoremAfterDiff} is an arbitrary positive integer, we can write residual term as $O\Big(\frac{1}{N^{(M+1)/2}}\Big)$ instead of $o\Big(\frac{1}{N^{M/2}}\Big).$


Together with expansions \eqref{eq:PetrovTheoremAfterDiff} and  \eqref{eq:s-diff_operator_asymp} and taking into the account that $h\sim N^{-1/2}$ this  results the following asymptotic expansion for $\Delta^{s} \rho(\hat{x}) $:
\begin{equation}
\label{eq:Delta_s_rho}
\begin{gathered}
\sqrt{N}\Delta^{s}_h\,\rho(\hat{x}(x)) = \sum\limits_{i=0}^{K} a_{s,i} \frac{d^{s+i}}{dx^{s+i}} \sqrt{N}\rho(\hat{x})h^{s+i}+ O({N^{-(K+s+1)/2}})  =\\
=\frac{e^{-x^2}}{\sqrt{2\pi}\sigma}\bigg(\sum\limits_{i=0}^{K} a_{s,i} \sum\limits_{\nu=0}^{M}\frac{\tilde{g}_{\nu,r}(x)}{N^{\nu/2}}h^{s+i}\bigg)+O(N^{-(K+s+1)/2})=\\
=\frac{e^{-x^2}}{\sqrt{2\pi}\sigma}\bigg(\sum\limits_{i=0}^{K} a_{s,i} \sum\limits_{\nu=0}^{M-i}\frac{\tilde{g}_{\nu,r}(x)}{N^{\nu/2}}h^{s+i}\bigg)+O(N^{-(K+s+1)/2})
\end{gathered}
\end{equation}
where $K=M+1,$ $x=O(1)$ and  $\hat{x}(x) = Np+(2Npq)^{1/2}x.$
This expression is written in fact in terms of the Hermite polynomials, but  it also  could be  rewritten using the parabolic cylinder functions $D_n(x)$  due to their close relation  with the  Hermite polynomials: $D_n(x) = 2^{-n/2}e^{-x^2/4}H_n(\frac{x}{\sqrt{2}}), n\in\mathbb{N}\cup\{0\}$.


We denote by $\psi_{s}^K(x)$ the sum $\sum\limits_{i=0}^{K} a_{s,i} \sum\limits_{\nu=0}^{K-1-i}\frac{\tilde{g}_{\nu,s+i}(x)}{N^{\nu/2}}h^{s+i} $ where coefficients $a_{s,j}$, with $j$ and $s\in\mathbb{N}\cup\{0\},$ are defined by expression \eqref{eq:forward_h_diff_asymp}.

Using representation \eqref{eq:RodriguesRewrite} together with \eqref{eq:Delta_s_rho}  we  obtain our main result providing asymptotic expansion of the Krawtchouk polynomials  in terms of the Hermite polynomials:

\begin{Theorem}Let $M$ and $n$ be  nonnegative integers and $A$ be a positive real. Let $\hat{x} - Np = (2Npq)^{1/2}x $. Then the following asymptotic expansion holds uniformly in $|x|\leq A$:
\label{Th:asympInTermsOfHermite}
\begin{equation}
\label{eq:mainTheorem}\rho(\hat{x}) k_n^{(p)}(\hat{x})  =\frac{e^{-x^2}}{\sqrt{2\pi N }\sigma} \frac{(-q)^n}{n!}\sum\limits_{k=0}^{M}\binom{n}{k}n^{\underline{k}}\,(\hat{x}+n-k)^{^{\underline{(n-k)}}}
\psi_{n-k}^M(x)+O(N^{\frac{n-M-2}{2}}).\end{equation}
\end{Theorem}
\begin{proof}
Since $h\sim N^{-1/2}$ we see that $\frac{e^{-x^2}}{\sqrt{2\pi N }\sigma}\psi_s^M(x)=O(N^{-(s+1)/2})$ and $(\hat{x}+s)^{^{\underline{s}}} = O(N^s)$
and therefore $(\hat{x}+s)^{^{\underline{s}}}\frac{e^{-x^2}}{\sqrt{2\pi N }\sigma}\psi_s^M(x) = O(N^{(s-1)/2}). $

Formula \eqref{eq:mainTheorem} follows directly from Rodrigues formula \eqref{eq:RodriguesRewrite} by changing $\Delta_h^s$ by their approximations $\frac{e^{-x^2}}{\sqrt{2\pi N }\sigma}\psi_s^M(x)$ obtained in \eqref{eq:Delta_s_rho}.

\end{proof}
In the special case $M=0$ Theorem \ref{Th:asympInTermsOfHermite} reduces to the well-known result:
\[k_n^{(p)}(\hat{x}) =\Big(\frac{Npq}{2}\Big)^{n/2}\frac{H_n(x)}{n!} +o(N^{\frac{n}{2}}),
\] stated above in \eqref{eq:classicalApproximOfKrawByHermite} (we used $\rho(\hat{x}) = \frac{1}{\sqrt{2\pi N}\sigma}e^{-x^2}+o(1)$ from Theorem \ref{Th:Petrov}~here).

In the general case expression \eqref{eq:mainTheorem}, in fact,  provides expansion\footnote{Of course, it is sensible to take $M\leq n$. } \eqref{eq:aympExpansion} (i.e. its right-hand-side comprises only powers of $N$ and $v=x\sqrt{2Npq}$, not of $\sqrt{N}\sim h^{-1}$). To see this note, that depending on parity of index $n$ Hermite polynomials $H_n$  contain either  odd or even powers only, expression \eqref{eq:q_nuDef} does not change parity of indexes of Hermite polynomials while
\eqref{eq:DExp(x)Hn} implies expansion \eqref{eq:s-diff_operator_asymp} also preserves parity of powers.

In the special case of $M=2$  and $v=\hat{x}-Np = o({N}^{1/3})$ we are going to obtain a more natural expression then (slightly cumbersome) expression \eqref{eq:mainTheorem} suggested by Theorem \ref{Th:asympInTermsOfHermite}.

First of all note that for the Hermite polynomials we have the following representations (see e.g. \cite{NikiforovSuslovUvarov}):
\begin{equation*}
\begin{gathered}H_{2l}(x) = (-1)^l2^l(2l-1)!!\big(1+\sum\limits_{j=1}^l\frac{4^j(-l)^{\bar{j}}}{(2j)!}x^{2j}\big), \\ H_{2l+1}(x) = (-1)^l2^{l+1}(2l+1)!!\big(x+\sum\limits_{j=1}^l\frac{4^j(-l)^{\bar{j}}}{(2j+1)!}x^{2j+1}\big), \end{gathered}\end{equation*}
where $l$ is a nonnegative integer. Additional assumption  $x\rightarrow0$  provides asymptotic expansions:
\begin{equation}
\label{eq:asympHermite}\begin{gathered}
H_{2l}(x) = (-1)^l2^l(2l-1)!!\big(1-2lx^2\big)+o(x^2),\ l =1, 2\dots ,\\ H_{2l+1}(x) = (-1)^l2^{l+1}(2l+1)!!x+o(x^2), \ l = 0,1,\dots\end{gathered}\end{equation}

For any nonnegative integers $l, m$ and  $C$ we have the following asymptotic expansion for the falling factorial $ (m+C)^{\underline{l}}$ as $m\rightarrow\infty$:

\begin{equation}\label{eq:asympFalling} (m+C)^{\underline{l}} = m^l+\Big(lC- \frac{l(l-1)}{2}\Big)m^{l-1}+O(m^{l-2}). \end{equation}

\begin{Corollary}
\label{Cor:asympSimplified}
For any sequence $\varepsilon(N)$ such that   $\lim\limits_{N\rightarrow\infty}\varepsilon(N)=0,$
the following asymptotic expansions hold uniformly in $v$ satisfying $|v|\leq\varepsilon(N)N^{1/3}$ :
\[k_{2l}^{(p)}(Np+v)= (-1)^l\frac{(2l-1)!!\big(pqN\big)^l}{(2l)!}\bigg(1-\frac{9v^2+t_1v+t_2}{9pq N}l
\bigg)+o(N^{l-1}),\]
where $t_1 =6(p-\frac12)(4l-1) $ and $t_2 = (l-1)\big(1+4l+(16l-5)pq\big)$ and $l \in \mathbb{N};$
\[k_{2l+1}^{(p)}(Np+v)=  (-1)^l\frac{(2l-1)!!\big(pqN\big)^l}{(2l)!} \frac{4l(p-\frac12)+3v}{3}+o(N^{l}) ,\]
where $l \in\mathbb{N}\cup\{0\}. $
\end{Corollary}
\begin{proof}
We denote by $\psi(x)$ the function $\frac{\sqrt{2\pi N }\sigma}{e^{-x^2}}$.
Let $k$ be  a nonnegative integer. Applying Theorem \ref{Th:Petrov} for $M=2$ and identity \eqref{eq:DExp(x)Hn}, we obtain
\[\begin{gathered}\psi(x)\frac{d^k}{dx^k}\rho(\hat{x}) =(-1)^k\Big( H_{k}(x)+\frac{\gamma_3}{2^{3/2}\sigma^3}\frac{H_{3+k}(x)}{3!\sqrt{N}}+\\+\frac{\frac{\gamma_4}{\sigma^4}H_{4+k}(x)+\frac{1}{3!}\big(\frac{\gamma_3}{\sigma^3}\big)^2 H_{6+k}(x) }{4\cdot4!N}\Big)+o\Big(\frac{1}{N}\Big),\end{gathered}\]
where
$\gamma_3 =pq_{ }(1-2p),\, \gamma_4 =-pq_{ }(6pq-1). $\  
Under the additional assumption\footnote{The slightly weaker assumption $v=O({N}^{1/3})$ would already need additional term $\frac{1}{6}\frac{v^3(q-p)}{(pq)^2N^2}$ in the right-hand-side of \eqref{eq:M(v)simplified}.} $v:=\hat{x}-Np=o({N}^{1/3})$ we get a simpler expression:
\begin{equation}
\label{eq:M(v)simplified}\psi(x)\rho\Big(\frac{v}{\sqrt{2 N } \sigma}\Big) = 1 -\frac{1-pq-6v(p-q)}{12pqN}+o\Big(\frac{1}{{N}}\Big).\end{equation}
In the same manner for odd-order derivatives we have
\begin{equation*}
\psi(x)\frac{d^{2l+1}}{dx^{2l+1}}\rho(\hat{x}) = (2l-1)!!(-2)^l\Big(1+\frac{36lv^2+\tau_1v+\tau_2}{36pq N}\Big)+o(\frac{1}{N}),
\end{equation*}
where $\tau_1 = 6(1-2p)(2l+3)(2l+1),$ $\tau_2 = (2l+1)(2l+3)(1+l-(1+4l)pq)$, and for even-order
\begin{equation*}
\begin{gathered}
\psi(x)\frac{d^{2l}}{dx^{2l}}\rho(\hat{x}) = (-1)^{l+1}(2l+1)!!2^l\frac{\sqrt{2}}{\sqrt{Npq}}(v+\frac{(1-2p)(2l+3)}{6})+o(\frac{1}{N}),
\end{gathered}
\end{equation*}
with $l\in\mathbb{N}$.
After some algebra using expansions \eqref{eq:asympHermite} and \eqref{eq:asympFalling}  and ignoring terms of the order lower than  $[\frac{n-1}{2}]$ we  obtain from \eqref{eq:mainTheorem} the required asymptotic expansion (and it naturally depends on the parity of $n$).
\end{proof}

Taking higher-order asymptotic expansions in \eqref{eq:asympHermite} and \eqref{eq:asympFalling} and taking appropriately high value of $M$ one can easily obtain other values of $c_j,$ $j\geq1,$ from expansion \eqref{eq:aympExpansion}.

We conclude this paper with a slight modification of  Corollary \ref{Cor:asympSimplified} for the function $K_{n}(x,p, N_1)$, with  $N_1 = N-i,$ $v=x-Np=o({N}^{1/3})$ and $i = O(1)$.

\begin{Corollary}
\label{Cor:specialAsymp}
For any sequence $\varepsilon(N)$ such that   $\lim\limits_{N\rightarrow\infty}\varepsilon(N)=0,$
the following asymptotic expansions hold uniformly in $v$ satisfying $|v|\leq\varepsilon(N)N^{1/3}$:
\[K_{2l}(x,p, N-i)= \Big(-\frac{q}{p}\Big)^l\frac{(2l-1)!!}{N^l}\bigg(1-\frac{9(v+ip)^2+\tilde{t}_1(v+ip)+\tilde{t}_2}{9pq N}l
\bigg)+o\big({N^{-l-1}}\big),\]
where $\tilde{t}_1 =6(p-\frac12)(4l-1) $ and $\tilde{t}_2 = (l-1)\big(1+4l+(16l-5)pq\big)-9pq(i+2l-1)$ and $l \in \mathbb{N}.$
\[K_{2l+1}(x,p, N-i)=  \Big(-\frac{q}{p}\Big)^l\frac{(2l+1)(2l-1)!!}{N^{l+1}}\cdot \frac{4l(p-\frac12)+3(v+ip)}{3p}+o(N^{-l-1}) ,\]
where $l \in\mathbb{N}\cup\{0\}. $
\end{Corollary}
\begin{proof} We have
$K_{n}(x,p, N_1) =K_{n}(Np+v,p, N_1) = K_{n}(N_1p+v_1,p, N_1)  = (-p)^n\binom{N_1}{n} \ k_{n}^{(p)}(N_1p+v_1, N_1),  $
where $v_1 = v+ip$. Using \eqref{eq:asympFalling} together with expression $(N-i)^{-l}=(N)^{-l}+il\,(N)^{-l-1}+O(N^{-l-2})$ we easily obtain the required asymptotic expansions from Corollary \ref{Cor:asympSimplified}.
\end{proof}
The uniform asymptotic expansions from Corollaries \ref{Cor:asympSimplified} and  \ref{Cor:specialAsymp} are obtained under the assumption $|v|  \leq\varepsilon(N)N^{1/3}.$ We can also use the preceding asymptotic expansions for $|v|  \leq\varepsilon(N)N^{1/2},\, \varepsilon(N)\rightarrow0$. However, if $N^{1/3}=O(v)$ we have to multiply the estimations of the residual terms of these asymptotic expansions by $\sqrt{N}$.


\section{Appendix}
In this section we prove formula \eqref{eq:PetrovTheoremAfterDiff}.
We denote by $S_N^M(\xi) $ such function that
\[\sqrt{N}\rho(\hat{\xi}) = e^{S_N^M(\xi)}\phi^M(\xi),\]
where $\phi^M(\xi) = \frac{1}{\sqrt{2\pi}\sigma}e^{-\xi^2}\, \sum\limits_{\nu=0}^{M}\frac{\tilde{q}_{\nu}(\xi)}{N^{\nu/2}}$ and polynomials $\tilde{q}_\nu$ are defined in Theorem 1 and, as above,  $\hat{s} = Np+\sqrt{2 N pq} s$. We denote by  $\phi^M_s(\xi)$ the  derivative $\frac{d^s}{d\xi^s}\phi^M(\xi) =(-1)^s\frac{1}{\sqrt{2\pi}\sigma}e^{-\xi^2}\, \sum\limits_{\nu=0}^{M}\frac{\tilde{q}_{\nu+s}(\xi)}{N^{\nu/2}}$. Using the integral Cauchy formula we obtain
\[\sqrt{N}\frac{d^s}{dx^s}\rho(\hat{x}(x)) = \frac{\sqrt{N}s!}{2\pi i}\int\limits_{C_x}\frac{\rho(\hat{\xi})d\xi}{(x-\xi)^{s+1}}=\]
\[=\frac{s!}{2\pi i}\bigg[\int\limits_{C_x} \frac{\phi^M(\xi)d\xi}{(x-\xi)^{s+1}} + \int\limits_{C_x}\frac{\big(\sqrt{N}\rho(\hat{\xi}) -\phi^M(\xi)\big)d\xi}{(x-\xi)^{s+1}}\bigg] =\]
\[ = \phi_s^M(x)+\frac{s!}{2\pi i} \int\limits_{C_x}\frac{\phi^M(\xi)\big(e^{S_N^M(\xi)}-1\big)d\xi}{(x-\xi)^{s+1}}\]
where $C_x$ is a closed contour around point $x$. Somewhat arbitrary we set  $C_x$ to be a unit circle around point $x$: \[C_x = \{\xi\, | \,\xi =x+ e^{i\varphi}, \varphi\in[0,2\pi]\}.\]

\begin{Lemma} Let $M$ be a nonnegative integer and $A$ be a positive real. Then for  $x\in[-A,A]$ and $\xi\in C_x$ we have the uniform estimate:
\[\big|e^{S_N^M(\xi)}-1\big| = o(N^{-M/2}).\]
\end{Lemma}
\begin{proof}
The Stirling expansion for the Gamma function (see e.g. book \cite{Fedoruk}, p.~34, Example 1.2., or \cite{AbramovitsStigan}, p.~83) could be written as follows:
\begin{equation}\label{Gamma:approximation}\ln \Gamma(z) =  (z-\frac12)\ln(z)-z+\frac{\ln(2\pi)}{2}+F_m(z)+O(|z|^{-2m-1}),\end{equation}
where $|z|\rightarrow\infty, \ |\text{arg}\, z|\leq \pi-\varepsilon<\pi$, $F_m(z)= \sum_{k=1}^{m} \frac{ B_{2k}}{2k(2k-1)z^{2k-1}}$ and $B_k,\ k\geq 1,$  are Bernoulli numbers. Using this expression it was shown in \cite{Sharapudinov}, formula~(17), that
\[\ln\rho({z}) = -\frac{N}{2pq}\big(\frac{z}{N}-p\big)^2-\frac12\ln(2\pi N pq)+S_N^0\Big(\frac{z-Np}{\sqrt{2Npq}}\Big),  \ z\rightarrow\infty, \]
where for arbitrary integer $m\geq0$ we have\footnote{This, of course, follows from estimate of the residual in \eqref{Gamma:approximation} see details, e.g. in \cite{AbramovitsStigan}.} $S_N^0(\xi) =  \Phi_m(\xi)+ O(N^{-2m-1})$
and the function $\Phi_m(\xi)$ is defined by the identity
\[\Phi_m(\xi) = F_m(N) - F_m(\hat{\xi})-F_m(N-\hat{\xi}) - Nr\Big(\frac{\hat{\xi}}{N}\Big) - \frac12D\Big(\frac{\hat{\xi}}{N}\Big),\]
with  $r(\tau) = \tau\ln\frac{\tau}{p}+(1-\tau)\ln\frac{1-\tau}{q}-\frac{1}{2pq}(\tau-p)^2,$\,
  $D(\tau) = \ln(1+\frac{(p-\tau)(\tau-q)}{pq})$. The function $\Phi_m(\xi)$  is analytic in the union $\bigcup_{x} B_x$ of unit balls $B_x$ bounded by the circles $C_x, \ x\in[-A,A]$.
For the function $S_N^0(\xi)$ it was shown  in \cite{Sharapudinov}, formula~(27), that $|S_N^0(\xi)| \leq \frac{C}{\sqrt{N}}$ for $\xi\in C_x$ and $x\in[-A,A]$. Since we know from Theorem~\ref{Th:Petrov} that at any point $x$ such that $\hat{x}=Np+(\sqrt{2Npq}x)\in \mathbb{Z}$  it holds that \[\big|e^{S_N^0(x) } -\sum\limits_{\nu=0}^{M}\frac{\tilde{q}_{\nu}(x)}{N^{\nu/2}}\big| = o(N^{-M/2}),\]
therefore it is sufficient to show that \[\big|e^{\Phi_{m}(\xi) } -\sum\limits_{\nu=0}^{M}\frac{\tilde{p}_{\nu}(\xi)}{N^{\nu/2}}\big| = o(N^{-M/2})\]
for some chosen $m=m(M),$ $\xi\in C_x$ and $x\in[-A,A]$ and \emph{some} polynomials $\{\tilde{p}_{\nu}(\xi)\}_{\nu=0}^M$. If it is already shown, then making $N$ large enough we see that $\tilde{p}_{\nu}(\xi) = \tilde{q}_{\nu}(\xi), 0\leq\nu\leq M,$  for all $\xi$ due to their coincidence in at least $[\sqrt{N}]$ points. Analyzing functions $F_m(\hat{\xi}), r(\frac{\hat{\xi}}{N}), D(\frac{\hat{\xi}}{N})$ as functions of variable $\frac{1}{\sqrt{N}}$ and parameter $x$ we see that for each of these functions we can write the Taylor series\footnote{We present initial terms of these series in the Remark~2 below.} as $\frac{1}{\sqrt{N}}\rightarrow 0$ of the form: $\sum\limits_{j=0}^{\infty}\frac{c_j}{(\sqrt{N})^j}\xi^{j+h}$ for some integer $h$.  Since $x\in[-A,A]$ we can truncate  each  series and write  $\sum\limits_{j=0}^{m}\frac{c_j}{(\sqrt{N})^j}x^{j+h}+f(N)$, where $|f(N)|\leq\frac{C}{N^{m}}$ uniformly for $\xi\in \bigcup_{x} B_x$. It is sufficient to set $m$ to be equal to $2M$. In the same manner truncating the Taylor series for the function $e^{\Phi_m(\xi)}$ we see that $|e^{\Phi_m(\xi)}- \sum\limits_{\nu=0}^{m}\frac{\tilde{p}_{\nu}(\xi)}{N^{\nu/2}}|=o(\frac{1}{N^{-M/2}}),$ for $p_0(\xi)\equiv1$ and some polynomials $p_\nu(\xi)$, where $\xi\in C_x $ and $x\in[-A,A]$.

 Since  $S_N^M(\xi)$ equals to $S_N^0(\xi)-  \Phi_m(\xi)+ o(N^{-m})$ we get that $|S_N^M(\xi)|\leq  o(N^{-M/2})$ and using the Lagrange Theorem we obtain that  $\big|e^{S_N^M(\xi)}-1\big| = o(N^{-M/2}),$ for $\xi\in C_x$ and $x\in[-A,A].$
\end{proof}

\fontsize{9}{9}
\selectfont

\emph{Remark $2$.} It is possible to obtain estimate $|e^{\Phi^M(\xi)} -\sum\limits_{\nu=0}^{M}\frac{\tilde{q}_{\nu}(\xi)}{N^{\nu/2}}| = o(N^{-M/2})$ without usage of Theorem 1 but directly from the analysis of the function $S_N^0$ at least for small values of $M$. For example, we can obtain the following asymptotic expansions, $N\rightarrow\infty$:

\[\begin{gathered}
r\Big(\frac{\hat{x}}{N}\Big) = -\frac{\sqrt{2}}{3}\,{\frac {\left( 2
\,p-1 \right) }{ \sqrt { \left( 1 -p\right) p} }}
\frac{ {x}^{3}}{{N}^{3/2}} +O \left( {
N}^{-2} \right), \quad
 D\Big(\frac{\hat{x}}{N}\Big)=\frac {\sqrt {2}\left( 2p-1 \right) }{ \sqrt {
 \left( 1-p \right) p}}\frac{x}{{N}^{1/2}}
+O \left( {N}^{-1} \right).
\end{gathered}\]
For the function $F_m$ with $m=2$ we have:
$\begin{gathered}F_2( z)= 1/ \left( 12 z\right). \end{gathered}$
Analogously we obtain expansions:
\[\begin{gathered}F_2( \hat{x})={\frac {1}{12Np}}+O \left( {N}^{-3/2}
 \right)
, \quad F_2( N-\hat{x})=\frac {1}{ 12\left( 1-p \right) N}+O \left( {N}^{-3/2}
 \right),
\end{gathered}\]
\[\begin{gathered}F_2( N)=1/12\,{N}^{-1}+O(N^{-3/2}).
\end{gathered}\]
Therefore we can write the first term of the asymptotic expansion of the function $e^{\Phi_M(x)}$ taking only the initial terms of the above expressions:
\[e^{\Phi_M(x)} = \frac{(1-2p)}{2^{3/2}(pq)^{1/2}}\frac{8x^3-12x}{3!\sqrt{N}}+O\Big(\frac{1}{N}\Big)=\frac{\tilde{q}_{1}(x)}{N^{1/2}}+O\Big(\frac{1}{N}\Big).\] This gives an explicit proof of Lemma 1 for $M=1$. It is an interesting question whether there is any  short and explicit proof of Lemma 1 for all values of parameter $M$ that does not use Theorem~1?

\fontsize{11}{11}
\selectfont

\begin{Lemma} Let $M$ be a nonnegative integer and $A$ be a positive real. Then for $x\in[-A,A]$ and $N$ large enough  we have the uniform estimate:
\[\frac{s!}{2\pi i} \int\limits_{C_x}\Big|\frac{\phi^M(\xi)}{(x-\xi)^{s+1}}\Big||d\xi|= O(1).\]
\end{Lemma}
\begin{proof}
For a given value of $M$, all $x\in[-A,A]$   and $N=N(A,M)$  large enough we have
$|\sum\limits_{\nu=0}^{M}\frac{\tilde{q}_{\nu}(x+e^{i\varphi})}{N^{\nu/2}}| \leq 2$.

\[\int\limits_{C_x}\Big|\frac{\phi^M(\xi)}{(x-\xi)^{s+1}}\Big||d\xi| \leq2\int\limits_{0}^{2\pi}e^{-(x+2\cos(\varphi))^2/2} d\varphi<Ce^{-x^2/4}\]
and constant $C$ does not depend on $x$.
\end{proof}

Combining Lemma 1 and Lemma 2 we can estimate the second integral
\[\Big|\frac{s!}{2\pi i} \int\limits_{C_x}\frac{\phi^M(\xi)\big(e^{S_N^M(\xi)}-1\big)d\xi}{(x-\xi)^{s+1}} \Big|= o(N^{-M/2}), \]
for $\xi\in C_x, $ $x\in[-A,A],$ that finishes the proof.

\end{document}